\documentclass[11pt,twoside]{amsart}
\usepackage{mabliautoref}
\usepackage{amssymb,amsthm,amsmath,amsfonts}
\usepackage{stmaryrd}
\RequirePackage[dvipsnames]{xcolor}
\usepackage{hyperref}
\usepackage{mathtools}
\usepackage{comment}
\usepackage[abbrev,alphabetic]{amsrefs}
\usepackage[all]{xy}
\usepackage{tikz}
\usepackage{tikz-cd}

\usepackage{kantlipsum} 

\calclayout

\hypersetup{
	bookmarksdepth=3,
	bookmarksopen,
	bookmarksnumbered,
	pdfstartview=FitH,
	colorlinks,
	linkcolor=Sepia,
	anchorcolor=BurntOrange,
	citecolor=MidnightBlue,
	citecolor=OliveGreen,
	filecolor=BlueViolet,
	menucolor=Yellow,
	urlcolor=OliveGreen
}

\newcommand{\bQ}{\mathbf{Q}}

\newcommand{\bP}{\mathbf{P}}
\newcommand{\bA}{\mathbf{A}}

\newcommand{\bB}{\mathbf{B}}
\newcommand{\bF}{\mathbf{F}}

\newcommand{\cB}{\mathcal{B}}
\newcommand{\cC}{\mathcal{C}}
\newcommand{\cD}{\mathcal{D}}

\newcommand{\cG}{\mathcal{G}}

\newcommand{\cM}{\mathcal{M}}

\newcommand{\cO}{\mathcal{O}}

\newcommand{\cV}{\mathcal{V}}
\newcommand{\cW}{\mathcal{W}}
\newcommand{\cX}{\mathcal{X}}
\newcommand{\cY}{\mathcal{Y}}
\newcommand{\cZ}{\mathcal{Z}}

\newcommand{\supp}{\mathrm{Supp}}

\newcommand{\pr}{\textup{pr}}

\newcommand{\charac}{\textup{char}}

\DeclareMathOperator{\PPic}{\mathbf{Pic}}

\DeclareMathOperator{\Spec}{Spec}

\DeclareMathOperator{\relProj}{\textbf{\textup{Proj}}}

\DeclareMathOperator{\Sym}{\textup{Sym}}

\numberwithin{equation}{section}

\begin{document}

\title{Non-extendable MMPs}

\author{Iacopo Brivio}
\address{Center of Mathematical Sciences and Applications, Harvard University, 20 Garden St, Cambridge, MA 02138, USA}
\email{ibrivio@cmsa.fas.harvard.edu}




\dedicatory{Dedicated to James M\textsuperscript{c}Kernan on the occasion of his sixtieth birthday.}


\begin{abstract}
    We construct examples of families of pairs over a DVR of positive characteristic, with very mild singularities, such that the MMP on the closed fiber does not extend to a relative MMP.
\end{abstract}

\maketitle


\section{Introduction}

Let $k$ be a field and set $S\coloneq\Spec (k\llbracket s \rrbracket)$. We will denote by $s$ and $\eta$ the closed and generic point of $S$, respectively. If $X$ is an $S$-scheme, we will denote by $X_s$ and $X_\eta$ the corresponding base change. An analogous notation will be used for coherent sheaves, divisors, and morphisms. 

In \cite{Kol3foldcharp} Koll\'ar showed that, when $k$ is of characteristic $p>0$, there exist families $\varphi_X\colon (X,D)\to S$ of plt 3-fold pairs such that $K_X+D$ is big and semiample and \textit{asymptotic invariance of plurigenera} (AIP) fails, i.e. the restriction map
\begin{center}
    \begin{tikzcd}
        H^0(X,\omega^m_{X/S}(mD))\arrow[r] & H^0(X_s,\omega^m_{X_s/k}(mD_s))
    \end{tikzcd}
\end{center}
is \textit{not} surjective for all $m\geq 1$ divisible enough. Equivalently (see \autoref{l-eq_cond_AES_bigsa}), letting $\pi\colon X\to Y\coloneqq\relProj_S(R(K_X+D))$ the relative canonical model, its base change of $\pi$ to the closed fiber has a non-trivial Stein factorization
\begin{center}
    \begin{tikzcd}
        \pi_s\colon X_s\arrow[r,"\pi^\nu_s"] & (Y_s)^\nu\arrow[r,"\nu"] & Y_s,
    \end{tikzcd}
\end{center}
where $\pi_s^\nu$ is the canonical model of $(X_s,D_s)$ and the normalization $\nu$ is a universal homeomorphism which is small, i.e. isomorphic in codimension one. Roughly speaking, the canonical model of $(X_s,D_s)$ does not extend to the canonical model of $(X,D)$, but it does only after composing it with a suitable universal homeomorphism. In particular, $(Y,\pi_*D)\to S$ is a stable family with non-$S2$ central fiber, thus showing that the KSB-A moduli functor is no longer proper in characteristic $p>0$ and dimension $\geq 3$.

It is then natural to ask whether the same phenomenon can occur when running a relative MMP. That is, are there families of terminal/klt pairs $\varphi_X\colon (X,D)\to S$ and a $(K_X+D)$-MMP over $S$, $\varrho\colon X\dashrightarrow Y$, such that $\varrho_s$ is a $(K_{X_s}+D_s)$-MMP ``up to universal homeomorphism''? In this paper we answer this question affirmatively. We closely follow the construction of \cite{Kol3foldcharp}. The first step is constructing a Mori fiber space over $S$ whose restriction to the closed fiber factors through Frobenius.

\begin{theorem}[Non-extendable Mori fiber space]\label{t-ne_MFS}
    When $\charac(k)=2$, there exists a projective family of regular Fano 3-folds $\varphi_W\colon W\to S$ with a $K_W$-Mori fiber space structure $p_2\colon W\to \bP_S^1$ over $S$, such that the restriction $p_{2,s}$ has non-trivial Stein factorization
    \begin{center}
        \begin{tikzcd}
            W_s\arrow[r,"\bar{p}_{2,s}"] & (\bP^1_{k})^{(-1)}\arrow[r,"F_{\bP^1_{k}/k}"] & \bP^1_k,
        \end{tikzcd}
    \end{center}
    where $\bar{p}_{2,s}$ is a $K_{W_s}$-Mori fiber space.
\end{theorem}

By taking a suitable $\bP^1$-bundle over $W$ we obtain examples of non-extendable canonical models, divisorial contractions, and flips.

\begin{theorem}[Non-extendable canonical model]\label{t-ne_cm}
    When $\charac(k)=2$, there exists a projective family of terminal 4-fold pairs $\varphi_X\colon (X,D)\to S$, with regular fibers, such that $K_X+D$ is big and semiample over $S$, with relative canonical model $\pi_2\colon X\to Y_2$, whose restriction $\pi_{2,s}$ has non-trivial Stein factorization
    \begin{center}
        \begin{tikzcd}
            X_s\arrow[r,"\pi_{2,s}^\nu"] & (Y_{2,s})^\nu\arrow[r,"\nu"] & Y_{2,s},
        \end{tikzcd}
    \end{center}
    where $\pi_{2,s}^\nu$ is the canonical model of $K_{X_s}+D_s$ and the demi-normalization $\nu$ is a small universal homeomorphism. In particular 
    \[
    h^0\left(X_s,\omega_{X_s/k}^{m}(mD_s)\right)>h^0\left(X_\eta,\omega_{X_\eta/k(\!(s)\!)}^{m}(mD_\eta)\right)
    \]
    for all $m\geq 1$ divisible enough.
\end{theorem}

\begin{theorem}[Non-extendable divisorial contraction]\label{t-ne_dc}
    When $\charac(k)=2$, there exists a projective family of terminal 4-fold pairs of log general type $\varphi_X\colon (X,G)\to S$, with regular fibers, satisfying condition HNCC\footnote{See \autoref{d-HNCC}.}, and a $(K_X+G)$-negative divisorial contraction $\pi_2\colon (X,G)\to (Y_2,G_2\coloneqq\pi_{2,*}G)$ over $S$, such that $K_{Y_2}+G_2$ is nef and the restriction $\pi_{2,s}$ has a non-trivial Stein factorization
    \begin{center}
        \begin{tikzcd}
           X_s\arrow[r,"\pi_{2,s}^\nu"] & (Y_{2,s})^\nu\arrow[r,"\nu"] & Y_{2,s},
        \end{tikzcd}
    \end{center}
    where $\pi_{2,s}^\nu$ is a $(K_{X_s}+G_s)$-negative divisorial contraction and the demi-normalization $\nu$ is a small universal homeomorphism.
\end{theorem}

\begin{theorem}[Non-extendable flip]\label{t-ne_flip}
    When $\charac(k)=2$, there exists a projective family of klt 4-fold pairs of log general type $\varphi_{Y_1}\colon (Y_1,B_1)\to S$, satisfying condition HNCC, and a $(K_{Y_1}+B_1)$-flip over $S$
    \begin{center}
        \begin{tikzcd}
            (Y_1,B_1)\arrow[rr,"\phi",dashed]\arrow[dr] &  & (Y_2,B_2)\arrow[dl]\\
                           & C, & 
        \end{tikzcd}
    \end{center}
    such that $K_{Y_2}+B_2$ is nef, and $\phi_s$ factors as
    \begin{center}
        \begin{tikzcd}
            Y_{1,s}\arrow[r,"\phi_s^\nu",dashed] & (Y_{2,s})^\nu\arrow[r,"\nu"] & Y_{2,s},
        \end{tikzcd}
    \end{center}
    where $\phi_s^\nu$ is a $(K_{Y_{1,s}}+B_{1,s})$-flip, and the demi-normalization $\nu$ is a small universal homeomorphism.
\end{theorem}

\autoref{t-ne_MFS}, \autoref{t-ne_dc}, and \autoref{t-ne_flip} show that \cite{Bri_IPGMM}*{Theorem 4.3} is sharp. In particular \autoref{t-ne_flip} is interesting because in \cites{BCHM,HM'} (see also \cite{HM_cortisbook} for a more self-contained explanation) it is shown that flips exist in any dimension by ``lifting'' finite generation of a canonical ring from a prime divisor $T\subset X$. \autoref{t-ne_flip} (see also \autoref{t-ne_flip_pf}) shows that this can fail in positive characteristic. While this kind of pathological behavior is not new (see \cites{CT_plt,Ber}), to the best of our knowledge \autoref{t-ne_flip} is the first instance of this phenomenon where $T=X_s$ is a fiber of a morphism.

In \cites{Bri20,Kol3foldcharp} the proposed examples of families violating AIP all have non-zero boundary divisor. More recently we managed to construct examples of smooth families of varieties (i.e. without boundary) with semiample canonical divisor of Kodaira dimension one for which AIP fails (see \cite{Bri_IPGMM}*{Theorem 7.1}). The remaining question is then whether this phenomenon can occur for smooth families of general type varieties. By applying a standard cyclic covering construction to a suitable spreading out of the family in \autoref{t-ne_cm} we answer this question positively.

\begin{theorem}\label{t-failure_AIP_gt}
    There exist smooth families of 7-folds $\varphi_{\cV}\colon\cV\to\Spec( \bF_2\llbracket s\rrbracket)$ with a birational contraction $\widetilde{\bar{\zeta}}\colon \cV\to \bar{\cY}$ over $\bF_2\llbracket s\rrbracket$ such that $K_{\cV}\sim_{\bar{\cY},\bQ}0$, and $\bar{\cY}_s$ is not $S2$. In particular, the restriction map
    \begin{center}
        \begin{tikzcd}
            H^0\left(\cV,\omega_{\cV/\bF_2\llbracket s\rrbracket}^m\right)\arrow[r] & H^0\left(\cV_s,\omega_{\cV_s/\bF_2}^m\right)
        \end{tikzcd}
    \end{center}
    fails to be surjective for all $m\geq 1$ divisible enough.
\end{theorem}

\begin{remark}
    It remains an open question whether the kind of behavior explained in \autoref{t-failure_AIP_gt} can occur also over a mixed characteristic DVR.
\end{remark}

\begin{remark}
    We point out that, for families of surface pairs satisfying HNCC (\autoref{d-HNCC}), relative MMP always restricts to fiberwise MMP (\cite{BBS}*{2.6}). Similarly, for families of surface pairs with semiample canonical divisor, forming the relative canonical model commutes with restriction to a fiber, at least when the residue characteristic is $>5$ (\cite{ABP}*{Section 4}).
\end{remark}

\noindent
\textbf{Acknowledgment.} We are grateful to Fabio Bernasconi for helpful comments. We thank the Center of Mathematical Sciences and Applications (CMSA) for financial support.

\section{Preliminaries}
\subsection{Conventions and notation}\label{ss-conventions_notations}
All of our schemes will be (geometrically) integral, and quasi-projective over a field or DVR (typically $k,k(\!(s)\!)$, or $k\llbracket s\rrbracket$, respectively).

A morphism of schemes $f\colon X\to Y$ is a \textit{universal homeomorphism} if, for all morphisms $Y'\to Y$ the base-change $f_{Y'}\colon X\times_Y Y' \to Y'$ is a homeomorphism. A finite universal homeomorphism is said to be \textit{small} if it is an isomorphism in codimension one. Let $k$ be a field of positive characteristic: it is known that a finite morphism of $k$-schemes is a universal homeomorphism if and only if it factors a power of the $k$-linear Frobenius morphism (\cite{Kol_QuotSpaces}*{Proposition 6.6}). As a consequence, if $f\colon X\to Y$ is a universal homeomorphism of $k$-schemes, then $f$ induces an isomorphism
\begin{center}
    \begin{tikzcd}
        \textup{Pic}(Y)\left[\frac{1}{p}\right]\arrow[r,"f^*"] & \textup{Pic}(X)\left[\frac{1}{p}\right].
    \end{tikzcd}
\end{center}
In particular, when $X$ and $Y$ are projective $k$-schemes, $f$ identifies the spaces of (co)dimension one cycles with real coefficients modulo numerical equivalence on $X$ and $Y$. We will be tacitly using this fact throughout the paper.

\begin{enumerate}
    \item If $f\colon X\to Y$ is a morphism and $D$ is a $\bQ$-Cartier $\bQ$-divisor on $Y$, we will often write $D_X$ as shorthand for $f^*D$.
    \item We will often use the terms ``$\bQ$-Cartier $\bQ$-divisor'' and ``$\bQ$-line bundle'' interchangeably.
    \item A \textit{contraction} is a proper morphism of schemes $f\colon X\to Y$ such that $f_*\cO_X=\cO_Y$. Note that if $X$ is normal, then so is $Y$.
    \item A \textit{pair} $(X,D)$ consists of a normal scheme $X$, quasi-projective over $S$, and a $\bQ$-divisor $D\geq 0$ such that $K_X+D$ is $\bQ$-Cartier. When $X$ is a scheme over a field we further require it to be geometrically integral.
    \item We say $\varphi_X\colon (X,D)\to S$ is a \textit{family of pairs} (resp. a \textit{projective family of pairs}) if $(X,D)$ is a pair, the morphism $\varphi_X\colon X\to S$ is surjective (resp. a projective contraction) with normal and geometrically integral fibers, and $\varphi_X|_{D_i}\colon D_i\to S$ is flat for all irreducible components $D_i$ of $\supp(D)$. In particular, $(X_t,D_t)$ is a pair for all $t\in S$.
    \item A line bundle $L$ on a quasi-projective $S$-scheme $X$ is called \textit{ample} (resp. \textit{big}) if there exists an open immersion $j\colon X \hookrightarrow X'$ into a projective $S$-scheme $X'$ and an ample (resp. big) line bundle $L'$ on $X'$ such that $L \cong j^*L'$.
    \item A line bundle $L$ on a quasi-projective $S$-scheme $X$ is called \textit{semiample} if there exists an open immersion $j\colon X \hookrightarrow X'$ into a projective $S$-scheme $X'$ and a semiample line bundle $L'$ on $X'$ such that $L \cong j^*L'$ and, letting $f'\colon X'\to Z$ be the corresponding semiample contraction, the restriction $f'\coloneqq f\circ j$ is a contraction as well. We will refer to $f$ as the \textit{semiample contraction of $L$}.
    \item If $L$ is a semiample $\bQ$-divisor on a proper $S$-scheme $X$, a \textit{general element of $|L|_{\bQ}$} means a $\bQ$-divisor of the form $D/m$, where $D\in |mL|$ is general and $m\geq 1$ is divisible enough.
\end{enumerate}

We refer the reader to \cite{Kol_SMMP} for the definitions of the various classes of singularities of pairs appearing in the Minimal Model Program (terminal, canonical, klt, plt, lc).

\begin{definition}\label{d-HNCC}
    Let $\varphi_X\colon (X,D)\to S$ be a projective family of pairs. We say it satisfies \textit{condition HNCC}\footnote{``Horizontal Non-Canonical Centers''.} if, whenever $N\subset \bB_-(K_X+D)$ is a non-canonical center of $(X,D+X_s)$, then $N$ is $\varphi_X$-horizontal.
\end{definition}

\begin{remark}
    Condition HNCC is necessary for ensuring that, given a family of pairs, the relative MMP restricts to a fiberwise MMP (see \cite{KawEXT}*{Example 4.3}). In characteristic zero, it is also sufficient (\cite{HMX}*{Lemmas 3.1 and 3.2}). By \autoref{t-ne_MFS}, \autoref{t-ne_dc}, and \autoref{t-ne_flip}, this is not the case in positive characteristic.
\end{remark}

\begin{definition}\label{d-AES}
    Let $X$ be an integral normal scheme, let $\varphi_X\colon X\to S$ be a surjective morphism with normal and geometrically integral fibers, and let $L$ be a line bundle on $X$. We say $L$ satisfies \textit{condition AES}\footnote{``Asymptotic Extension of Sections''.} if, for all $m\geq 0$ divisible enough, the restriction map 
    \begin{center}
        \begin{tikzcd}
            H^0\left(X,L^m\right)\arrow[r] & H^0\left(X_s,L_s^m\right)
        \end{tikzcd}
    \end{center}
    is surjective for all $m\geq 0$ divisible enough. When $\varphi_X$ is proper, this is equivalent to $h^0(X_t,L_t^m)$ being independent of $t\in S$ for all such $m$. 
\end{definition}

\begin{lemma}\label{l-eq_cond_AES_bigsa}
    Let $X$ be an integral normal scheme, let $\varphi_X\colon X\to S$ be a quasi-projective surjective morphism with normal and geometrically integral fibers, and let $L$ be a big and semiample line bundle on $X$. Let $\pi\colon X\to Y$ be the semiample contraction of $L$. Then the following are equivalent:
    \begin{itemize}
        \item[(1)] $L$ satisfies AES;
        \item[(2)] $Y_s$ is normal;
        \item[(3)] $R^1\pi_*\cO_X$ is torsion-free over $S$.
    \end{itemize}
\end{lemma}

\begin{proof}
    The morphism $\pi\colon X\to Y$ is a projective contraction of normal quasi-projective integral schemes, hence $Y\to S$ is a flat morphism. By construction, there exists a $\bQ$-line bundle $A$ on $Y$ such that $L\sim_{\bQ} \pi^*A$. Let $m\geq 0$ be a sufficiently divisible integer, and consider the standard exact sequence
    \begin{center}
        \begin{tikzcd}
            0\arrow[r] & L^m\arrow[r,"\cdot s"] & L^m\arrow[r] & L^m_s\arrow[r] & 0.
        \end{tikzcd}
    \end{center}
    Taking global sections, we see that point (1) is equivalent to the map $H^1(X,L^m)\xrightarrow{\cdot s} H^1(X,L^m)$ being injective. By the Grothendieck spectral sequence and Serre vanishing, we have an isomorphism $H^1(X,L^m)\cong H^0(Y,R^1\pi_*\cO_X\otimes A^m)$. As $A$ is ample, the sheaf $R^1\pi_*\cO_X\otimes A^m$ is generated by global sections. In particular, as $A^m$ is invertible and $Y$ is flat over $S$, point (1) is equivalent to the injectivity of $R^1\pi_*\cO_X\xrightarrow{\cdot s}R^1\pi_*\cO_X$, i.e. point (3). To see the equivalence with point (2), consider the short exact sequence
    \begin{center}
        \begin{tikzcd}
            0\arrow[r] & \cO_X\arrow[r,"\cdot s"] & \cO_X\arrow[r] & \cO_{X_s}\arrow[r] & 0.
        \end{tikzcd}
    \end{center}
    As $\pi$ is a contraction, by pushing forward we obtain
    \begin{center}
        \begin{tikzcd}
            \cO_Y\arrow[r,"\cdot s"] & \cO_Y\arrow[r] & \pi_{s,*}\cO_{X_s}\arrow[r] & R^1\pi_*\cO_X\arrow[r,"\cdot s"] & R^1\pi_*\cO_X.
        \end{tikzcd}
    \end{center}
    Then we see that point (3) is equivalent to $\pi_s$ being a contraction, and this clearly implies that $Y_s$ is normal. Conversely, if $Y_s$ is normal, then $\pi_s$ is a contraction by Zariski's Main Theorem (\cite{EGA4}*{Lemme 8.12.10.I}).
\end{proof}

\subsection{Maddock's del Pezzo}
From now until the end of the paper we set $k\coloneqq\bF_{2}(\alpha_1,\alpha_2,\alpha_3)$.
\begin{theorem}\label{t-maddock's_dP}
    There exists a regular projective $k$-surface $M$ which is geometrically integral and
    \begin{itemize}
        \item $\omega^{-1}_{M/k}$ is ample,
        \item $\rho(M)=1$,
        \item $\PPic^0_{M/k}$ is smooth and one-dimensional,
        \item the Poincar\'e bundle $P$ satisfies $P^4\cong\cO_{M\times\PPic^0_{M/k}}$.
    \end{itemize}
\end{theorem}

\begin{proof}
    In \cite{Mad_irregDP} the author constructs a regular del Pezzo surface $M$ over $k$ with $h^1(M,\cO_M)=1$. From the construction, it follows that $M$ is universally homeomorphic to $\bP^2_k$, hence $\rho(M)=1$. By \cite{FGA}*{Theorem 9.5.4}, we have that $\PPic^0_{M/k}$ exists as a quasi-projective $k$-scheme. Since $H^2(M,\cO_M)=0$, we have that $\PPic^0_{M/k}$ is smooth and one-dimensional (\cite{FGA}*{Proposition 9.5.19}). The last point follows by \cite{BT_dP}*{Theorem 1.3}.
\end{proof}

\section{Construction}

\subsection{A special family of Fano threefolds}\label{ss-special_family_fano3fold} Let $M$ be Maddock's del Pezzo surface and let $S$ be the completion of $\PPic^0_{M/k}$ at the origin, so that $S=\Spec (k\llbracket s \rrbracket)$. Upon replacing $P$ by its square, we may assume that it is 2-torsion\footnote{This is not strictly necessary, but it will simplify some computations.}. We will abuse notation and still denote by $P$ its base-change to $M_S\coloneq M\times S$.

Consider $Z\coloneqq M_S\times _S \bP_S^1$ and let $L\coloneqq P\boxtimes\cO(1)$. Fix homogeneous coordinates $[u:v]$ on $\bP^1_S$ and consider the section $\sigma\coloneqq 1_{P^2}\boxtimes uv\in H^0(M_S\times_S \bP^1_S,L^2)$, where $1_{P^2}\in H^0(M_S,P^2)$ is a nowhere vanishing section. Let $W\coloneqq Z[\sqrt{\sigma}]$ be the associated 2-to-1 cyclic cover. It fits in the commutative diagram
\begin{equation}\label{e-MFS}
    \begin{tikzcd}
         & W\arrow[d,"\gamma"]\arrow[ddl,"p_1",swap, bend right=10]\arrow[ddr,"p_2",bend left=10] & \\
         & Z\arrow[dl,"\pr_1"]\arrow[dr,"\pr_2",swap] & \\
        M_S\arrow[dr] &  & \bP^1_S\arrow[dl]\\
         & S. & 
    \end{tikzcd}
\end{equation}

A straightforward computation (see \cite{Bri20}*{3.1}) shows that the central fiber $W_s$ sits in the Cartesian square
\begin{equation}\label{e-cartesian_central_fiber}
    \begin{tikzcd}
        W_s\arrow[r,"\gamma_s"]\arrow[d]\arrow[dr,"p_{2,s}"] & Z_s\arrow[d,"\pr_{2,s}"]\\
        (\bP^1_k)^{(-1)}\arrow[r,"F_{\bP^1_k/k}",swap] & \bP^1_k.
    \end{tikzcd}
\end{equation}

We summarize the properties that will be relevant for us.
\begin{lemma}
    Keeping the above notation, the following hold.
    \begin{itemize}
    \item[(1)] The induced morphism $\varphi_W\colon W\to S$ is a flat projective contraction with regular and geometrically integral fibers.
    \item[(2)] The morphisms $p_1$ and $p_2$ are both contractions of relative Picard rank one.
    \item[(3)] The restriction $p_{1,s}\colon W_s\to M_s$ is a contraction.
    \item[(4)] The restriction $p_{2,s}\colon W_s\to \bP^1_k$ has a non-trivial Stein factorization.
    \item[(5)] $\omega_{W/S}^{-1}$ is ample.
\end{itemize}
\end{lemma}
\begin{proof}

    \begin{itemize}
        \item[(1)] Flatness of $\varphi_W$ is automatic (\cite{Har}*{Proposition 9.7}). Since $\gamma_*\cO_W=\cO_Z\oplus L^{-1}$ we also have $\varphi_{W,*}\cO_W=\cO_S$. As $M_s$ is regular and geometrically integral, so is $W_s$. As these properties are open in $S$ (\cite{EGA4}*{Th\'eor\`eme 12.2.1}), the same holds for $W_\eta$. 
        \item[(2)] Again, we have $\gamma_{*}\cO_W=\cO_Z\oplus L^{-1}$. As $\pr_i$ is a contraction and $\pr_{i,*}L^{-1}=0$, we have that $p_i$ is a contraction as well. Since $W$ is universally homeomorphic to $Z$ and $p_i$ to $\pr_i$, we have that $p_i$ have relative Picard rank one.
        \item[(3)] Let $H_1$ be a very ample divisor on $M_S$ and let $H_{1,W}$ be its pullback to $W$. By \cite{Bri_IPGMM}*{Lemma 3.5}, item (3) is equivalent to $H_{1,W}$ satisfying AES. This can be checked directly using the projection and K\"unneth formulae:
        \[
        \begin{split}
             h^0(W_t,H^m_{1,W_t})&=h^0(Z_t,\gamma_{t,*}\cO_{W_t}\otimes H^m_{1,Z_t})\\
                        &=  h^0(M_t,H^m_{1,t})\cdot h^0(\bP^1_t,\cO) +   h^0(M_t,P_t\otimes H^m_{1,t})\cdot h^0(\bP^1_t,\cO(1)) .
        \end{split}
        \]
        Provided that $H_1$ is sufficiently ample, all the terms in the above equation are independent of $t\in S$ for all $m\geq 0$.
        \item[(4)] This follows from \autoref{e-cartesian_central_fiber}, as $\pr_{2,s}$ is a contraction, the morphism $F_{\bP^1_k/k}$ is flat, and forming the Stein factorization commutes with flat base change (\cite{SP}*{Tag 03GX}).
        \item[(5)] By \autoref{e-cartesian_central_fiber} we have that $\omega_{W_s/k}^{-1}$ is ample. Since ampleness is an open condition, we conclude.
    \end{itemize}
\end{proof}

\subsection{M\texorpdfstring{\textsuperscript{c}}{c}Kernan's trick}\label{s-james_trick} Let $A\coloneqq \omega^{-1}_{W/S}$ and consider the projective bundle
\begin{center}
    \begin{tikzcd}
        X\coloneqq \relProj_W(\Sym^*(\cO_W\oplus A))\arrow[r,"\tau"] & W.
    \end{tikzcd}
\end{center}
It comes with two sections with image $W_0$ and $W_\infty$, having normal bundles $N_{W_0/X}=A^{-1}$ and $N_{W_\infty/X}=A$, respectively. In particular, the divisor $W_\infty$ is big and semiample, and its ample model is the projective cone of $W$ over $S$ with respect to $A$ (see \cite{Kol_SMMP}*{3.8})
\begin{center}
    \begin{tikzcd}
        X\arrow[r,"\beta"] & C\coloneqq C_p(W/S,A).
    \end{tikzcd}
\end{center}
The morphism $\beta$ is the resolution obtained by blowing up the $S$-point of $C$ corresponding to the vertex of the cone.
\begin{proposition}\label{p-AES_nAES_mckernantrick}
    Let $H_1$ and $H_2$ be sufficiently ample line bundles on $M_S$ and $\bP^1_S$, respectively, and denote by $H_{i,X}$ the pullback to $X$. Then the line bundles $H_{i,X}(W_\infty)$ are big and semiample over $S$ and, letting $\pi_i\colon X\to Y_i$ be the corresponding birational contraction, the following hold.
    \begin{enumerate}
        \item $\pi_{1,s}\colon X_s\to Y_{1,s}$ is a birational contraction; equivalently, $H_{1,X}(W_\infty)$ satisfies AES.
        \item $\pi_{2,s}$ has Stein factorization
        \begin{center}
            \begin{tikzcd}
            X_s\arrow[r,"\bar{\pi}_{2,s}"] & (Y_{2,s})^\nu\arrow[r,"\nu"] & Y_{2,s},         
            \end{tikzcd}
        \end{center}
        where the demi-normalization $\nu$ is a small universal homeomorphism which is not an isomorphism; equivalently, $H_{2,X}(W_{\infty})$ does not satisfy AES.
    \end{enumerate}
\end{proposition}

\begin{proof}
    As $W_\infty$ and the $H_{i,X}$ are semiample and the former is big, we have that the $H_{i,X}(W_\infty)$ are big and semiample. The ``equivalently'' part of items (a) and (b) follows by \cite{Bri_IPGMM}*{Lemma 3.5}. We now compute $h^0(X_t,H^m_{i,X_t}(mW_{\infty,t}))$ for all $m\geq 0$.
    \begin{enumerate}
        \item We have
        \begin{equation}
            \begin{split}
                h^0(X_t,H^m_{1,X_t}(mW_{\infty,t}))&=h^0(W_t,H^m_{1,W_t}\otimes\Sym^m(\cO_{W_t}\oplus A_t))\\
                                                   &=\sum_{j=0}^m h^0(W_t,H^m_{i,W_t}\oplus A^j_t),
            \end{split}
        \end{equation}
        hence it is enough to show that each of the summands is independendent of $t\in S$. Recall that
        \[
        \gamma_{t,*}\cO_{W_t}=\cO_{Z_t}\oplus L_t^{-1} \hspace{10mm} \textup{and} \hspace{10mm} \omega_{W_t/t}=\gamma_t^*(\omega_{Z_t/t}\otimes L_t),
        \]
        where $L=P\boxtimes \cO(1)$. Hence
        \[
        \gamma_{t,*} A_t^j=\left( (\omega^{-j}_{M_S/S}\otimes P^j)\boxtimes \omega^{-j}_{\bP^1_S/S}(-j) \right)\oplus \left( (\omega^{-j}_{M_S/S}\otimes P^{j+1})\boxtimes \omega^{-j}_{\bP^1_S/S}(-j-1) \right).
        \]

        By the projection and K\"unneth formulae we have
        \begin{equation}
            \begin{split}
                h^0(W_t, H^m_{1,W_t}\otimes A^j_t)&=h^0(Z_t,H^m_{1,Z_t}\otimes\gamma_{t,*}A^j_t)\\
                                                  &=h^0(Z_t, (\omega^{-j}_{M_t/t}\otimes P_t^j\otimes H^m_{1,t})\boxtimes \omega^{-j}_{\bP^1_t/t}(-j))+\\
                                                  &+h^0(Z_t,(\omega^{-j}_{M_t/t}\otimes P_t^{j+1}\otimes H^m_{1,t})\boxtimes \omega^{-j}_{\bP^1_t/t}(-j-1))\\
                                                  &=h^0(M_t,\omega^{-j}_{M_t/t}\otimes P_t^j\otimes H^m_{1,t})\cdot h^0(\bP^1_t,\omega^{-j}_{\bP^1_t/t}(-j))+\\
                                                  &+h^0(M_t,\omega^{-j}_{M_t/t}\otimes P_t^{j+1}\otimes H^m_{1,t})\cdot h^0(\bP^1_t,\omega^{-j}_{\bP^1_t/t}(-j-1)).
            \end{split}
        \end{equation}
    
    Looking at the last term in the equality, we see that the second factors of each summand are independent of $t\in S$ for any $j\geq 0$. But so are the first factors of each summand. Indeed, by Fujita vanishing we can pick $H_1$ positive enough so that $h^0(M_t,N_t\otimes P_t\otimes H_{1,t})=\chi(M_t,N_t\otimes P_t\otimes H_{1,t})$ for any nef line bundle $N$ on $M_S$. Hence $H_{1,X}(W_\infty)$ satisfies AES. 

    \item The same exact argument (with the obvious modifications) leads to the equality
    \begin{equation}
        \begin{split}
            h^0(W_t,H^m_{2,W_t}\otimes A_t^j)&=h^0(Z_t,H^m_{2,Z_t}\otimes\gamma_{t,*}A^j_t)\\
                                                  &=h^0(Z_t, (\omega^{-j}_{M_t/t}\otimes P_t^j)\boxtimes (\omega^{-j}_{\bP^1_t/t}(-j)\otimes H^m_{2,t}))+\\
                                                  &+h^0(Z_t,(\omega^{-j}_{M_t/t}\otimes P_t^{j+1})\boxtimes (\omega^{-j}_{\bP^1_t/t}(-j-1)\otimes H^m_{2,t}))\\
                                                  &=h^0(M_t,\omega^{-j}_{M_t/t}\otimes P_t^j)\cdot h^0(\bP^1_t,\omega^{-j}_{\bP^1_t/t}(-j)\otimes H^m_{2,t})+\\
                                                  &+h^0(M_t,\omega^{-j}_{M_t/t}\otimes P_t^{j+1})\cdot h^0(\bP^1_t,\omega^{-j}_{\bP^1_t/t}(-j-1)\otimes H^m_{2,t}).
        \end{split}
    \end{equation}
    In particular, for every $m\geq 0$, the first factor of the second summand for $j=0$ is just $h^0(M_t,P_t)$, which is \textit{not} independent of $t\in S$. We conclude by upper-semicontinuity of cohomology.
    \end{enumerate}
\end{proof}

The situation is summed up in the following commutative diagram

\begin{equation}\label{e-birl}
    \begin{tikzcd}
         & X\arrow[dl,"\pi_1",swap]\arrow[dr,"\pi_2"] & \\
        Y_1\arrow[dr,"\psi_1"]\arrow[ddr,"\varphi_{Y_1}",swap,bend right=10]\arrow[rr,"\phi",dashed] &  & Y_2\arrow[dl,"\psi_2",swap]\arrow[ddl,"\varphi_{Y_2}",bend left=10]\\
          & C\arrow[d] & \\
           & S, & 
    \end{tikzcd}
\end{equation}
where
\begin{itemize}
    \item $\pi_i$ is the semiample contraction of $H_{i,X}(W_\infty)$ and it is a divisorial contraction of relative Picard rank one for $i=1,2$;
    \item $\psi_i$ is a small contraction of relative Picard rank one for $i=1,2$;
    \item $\phi$ is a small proper birational map;
    \item $Y_{1,s}$ is normal;
    \item $Y_{2,s}$ is not $S2$.
\end{itemize}


\subsection{Excellent DVR} By ``spreading out'' Maddock's del Pezzo $\cM$, we can upgrade the statements of \autoref{t-ne_MFS}, \autoref{t-ne_cm}, \autoref{t-ne_dc}, and \autoref{t-ne_flip} to hold over a perfect DVR\footnote{See also \cite{BBF}*{Example 6.2}.}. More precisely, there exists a non-empty affine open set $U\subset\bA^3_{\bF_2}$ such that
\begin{itemize}
    \item there is a projective flat contraction $\cM\to U$, where the generic fiber is $M$ and the total space $\cM$ is smooth over $\bF_2$;
    \item there is a diagram
    \begin{equation}\label{e-MFS'}
    \begin{tikzcd}
         & \cW\arrow[d,"\widetilde{\gamma}"]\arrow[ddl,"\widetilde{p}_1",swap, bend right=10]\arrow[ddr,"\widetilde{p}_1",bend left=10] & \\
         & \cZ\arrow[dl,"\widetilde{\pr}_1"]\arrow[dr,"\widetilde{\pr}_2",swap] & \\
        \cM\otimes\bF_2\llbracket s\rrbracket\arrow[dr]\arrow[ddr,bend right=10] &  & \bP^1_U\otimes\bF_2\llbracket s\rrbracket\arrow[dl]\arrow[ddl,bend left=10]\\
         & U\otimes\bF_2\llbracket s\rrbracket\arrow[d] & \\
         & \Spec( \bF_2\llbracket s\rrbracket ),
    \end{tikzcd}
    \end{equation}
    where all the maps to $U\otimes \bF_2\llbracket s\rrbracket$ are flat and projective, the spaces $\cW$ and $\cZ$ are smooth over $\bF_2\llbracket s\rrbracket$, and base changing \autoref{e-MFS'} via $S\to U\otimes \bF_2\llbracket s\rrbracket$ (i.e. localizing $U\otimes \bF_2\llbracket s\rrbracket$ at the generic point of $U\otimes (\bF_2\llbracket s\rrbracket/(s))\cong U$) yields \autoref{e-MFS}.
    \item There is a diagram
    \begin{equation}\label{e-birl'}
        \begin{tikzcd}
            & \cX\arrow[dl,"\widetilde{\pi}_1",swap]\arrow[dr,"\widetilde{\pi}_2"] & \\
             \cY_1\arrow[dr,"\widetilde{\psi}_1",swap]\arrow[dddr,"\varphi_{\cY_1}",swap,bend right=10]\arrow[rr,"\widetilde{\phi}",dashed] &  & \cY_2\arrow[dl,"\widetilde{\psi}_2"]\arrow[dddl,"\varphi_{\cY_2}",bend left=10]\\
            & \cC\arrow[d] & \\
            & U\otimes \bF_2\llbracket s\rrbracket\arrow[d] &\\
            & \Spec( \bF_2\llbracket s\rrbracket )
        \end{tikzcd}
    \end{equation}
    where all the maps to $U\otimes \bF_2\llbracket s\rrbracket$ are flat projective contractions, the spaces $\cX,\cY_1$, and $\cY_2$ are smooth over $\bF_2\llbracket s\rrbracket$, and base changing \autoref{e-birl'} via $S\to U\otimes \bF_2\llbracket s\rrbracket$ yields \autoref{e-birl}. Similarly, there exists $\bQ$-divisors $\cD,\cG$ on $\cX$ and $\cB_i$ on $\cY_i$, such that
    \[
    (\cX,\cD)\times_{(U\otimes\bF_2\llbracket s\rrbracket)}S=(X,D)\hspace{5mm}(\cX,\cG)\times_{(U\otimes\bF_2\llbracket s\rrbracket)}S=(X,G),
    \]
    \[
    (\cY_1,\cB_1)\times_{(U\otimes\bF_2\llbracket s\rrbracket)}S=(Y_1,B_1)\hspace{5mm}(\cY_2,\cB_2)\times_{(U\otimes\bF_2\llbracket s\rrbracket)}S=(Y_2,B_2).
    \]
\end{itemize}

Then we obtain the following versions of Theorems \ref{t-ne_MFS} to \ref{t-ne_flip} over a perfect field. Their verification is straightforward and left as an exercise to the overly interested reader.

\begin{theorem}[Non-extendable Mori fiber space]\label{t-ne_MFS_pf}
    There exists a family of smooth 6-folds $\varphi_{\cW}\colon \cW\to \Spec( \bF_2\llbracket s \rrbracket )$ with a $K_{\cW}$-Mori fiber space structure $\widetilde{p}_2\colon \cW\to \bP_U^1\otimes\bF_2\llbracket s\rrbracket$ over $U\otimes\bF_2\llbracket s\rrbracket$, such that the restriction $\widetilde{p}_{2,s}$ has non-trivial Stein factorization
    \begin{center}
        \begin{tikzcd}
            \cW_s\arrow[r,"\bar{\widetilde{p}}_{2,s}"] & (\bP^1_{U})^{(-1)}\arrow[r,"F_{\bP^1_{U}/U}"] & \bP^1_U,
        \end{tikzcd}
    \end{center}
    where $\bar{\widetilde{p}}_{2,s}$ is a $K_{\cW_s}$-Mori fiber space.
\end{theorem}

\begin{theorem}[Non-extendable canonical model]\label{t-ne_cm_pf}
    There exists a family of terminal 7-fold pairs $\varphi_\cX\colon (\cX,\cD)\to \Spec( \bF_2\llbracket s\rrbracket )$, with smooth fibers, with a flat and projective contraction $\cX\to U\otimes\bF_2\llbracket s\rrbracket$ of $\bF_2\llbracket s\rrbracket$-schemes, such that $K_\cX+\cD$ is big and semiample over $U\otimes\bF_2\llbracket s\rrbracket$, with relative canonical model $\widetilde{\pi}_2\colon \cX\to \cY_2$ over $U\otimes\bF_2\llbracket s\rrbracket$, whose restriction $\widetilde{\pi}_{2,s}$ has a non-trivial Stein factorization
    \begin{center}
        \begin{tikzcd}
            \cX_s\arrow[r,"\widetilde{\pi}_{2,s}^\nu"] & (\cY_{2,s})^\nu\arrow[r,"\nu"] & \cY_{2,s},
        \end{tikzcd}
    \end{center}
    where $\widetilde{\pi}_{2,s}^\nu$ is the canonical model of $K_{\cX_s}+\cD_s$ over $U$, and the demi-normalization $\nu$ is a small universal homeomorphism. In particular, the restriction map
    \begin{center}
        \begin{tikzcd}
            H^0\left(\cX,\omega_{\cX/\bF_2\llbracket s\rrbracket}^{m}(mD) \right)\arrow[r] & H^0\left(\cX_s, \omega_{\cX_s/\bF_2}^{m}(mD_s) \right)
        \end{tikzcd}
    \end{center}
    is not surjective for all $m\geq 1$ divisible enough.
\end{theorem}

\begin{theorem}[Non-extendable divisorial contraction]\label{t-ne_dc_pf}
    There exists a family of terminal 7-fold pairs $\varphi_\cX\colon (\cX,\cG)\to \Spec(\bF_2\llbracket s\rrbracket)$, with smooth fibers, satisfying condition HNCC, with a flat and projective contraction $\cX\to U\otimes\bF_2\llbracket s\rrbracket$ of $\bF_2\llbracket s\rrbracket$-schemes, and a $(K_\cX+\cG)$-negative divisorial contraction $\widetilde{\pi}_2\colon (\cX,\cG)\to (\cY_2,\cG_2\coloneqq\widetilde{\pi}_{2,*}\cG)$ over $U\otimes\bF_2\llbracket s\rrbracket$, such that $K_{\cX}+\cG$ is big over $U\otimes\bF_2\llbracket s\rrbracket$, $K_{\cY_2}+\cG_2$ is nef over $U\otimes\bF_2\llbracket s\rrbracket$, and the restriction $\widetilde{\pi}_{2,s}$ has a non-trivial Stein factorization
    \begin{center}
        \begin{tikzcd}
           \cX_s\arrow[r,"\widetilde{\pi}_{2,s}^\nu"] & (\cY_{2,s})^\nu\arrow[r,"\nu"] & \cY_{2,s},
        \end{tikzcd}
    \end{center}
    where $\widetilde{\pi}_{2,s}^\nu$ is a $(K_{\cX_s}+\cG_s)$-negative divisorial contraction over $U$ and the demi-normalization $\nu$ is a small universal homeomorphism.
\end{theorem}

\begin{theorem}[Non-extendable flip]\label{t-ne_flip_pf}
    There exists a family of 7-fold pairs $\varphi_{\cY_1}\colon (\cY_1,\cB_1)\to \Spec( \bF_2\llbracket s\rrbracket )$, with klt fibers, satisfying condition HNCC, with a flat and projective contraction $\cY_1\to U\otimes\bF_2\llbracket s\rrbracket$ of $\bF_2\llbracket s\rrbracket$-schemes, and a $(K_{\cY_1}+\cB_1)$-flip over $U\otimes\bF_2\llbracket s\rrbracket$
    \begin{center}
        \begin{tikzcd}
            (\cY_1,\cB_1)\arrow[rr,"\widetilde{\phi}",dashed]\arrow[dr] &  & (\cY_2,\cB_2)\arrow[dl]\\
                           & \cC & 
        \end{tikzcd}
    \end{center}
    such that $K_{\cY_1}+\cB_1$ is big over $U\otimes\bF_2\llbracket s\rrbracket$, $K_{\cY_2}+\cB_2$ is nef over $U\otimes\bF_2\llbracket s\rrbracket$, and $\widetilde{\phi}_s$ factors as
    \begin{center}
        \begin{tikzcd}
            \cY_{1,s}\arrow[r,"\widetilde{\phi}_s^\nu",dashed] & (\cY_{2,s})^\nu\arrow[r,"\nu"] & \cY_{2,s},
        \end{tikzcd}
    \end{center}
    where $\widetilde{\phi}_s^\nu$ is a $(K_{\cY_{1,s}}+\cB_{1,s})$-flip over $U$, and the demi-normalization $\nu$ is a small non-trivial universal homeomorphism.
\end{theorem}


\section{Proofs}

\begin{proof}[Proof of \autoref{t-ne_MFS}]
    Just consider the family of Fano 3-folds $\varphi_W\colon W\to S$ together with the fibration $p_2\colon W\to\bP^1_S$ over $S$ constructed in \autoref{ss-special_family_fano3fold}. 
\end{proof}

\begin{proof}[Proof of \autoref{t-ne_cm}]
    We consider the contraction of $S$-schemes $\pi_2\colon X\to Y_2$ constructed in \autoref{s-james_trick}. The point is realizing it as the semiample contraction of $K_X+D$ for some effective $\bQ$-divisor $D$ such that $(X,D)$ and $(X_t,D_t)$ are terminal for all $t\in S$. Let $\Lambda\in |-K_{W/S}|_{\bQ}$ and $\Pi_2\in |H_2|_{\bQ}$ be a sufficiently general divisors so that $(W,\Lambda+\Pi_2)$ and $(W_t,\Lambda_t+\Pi_{2,t})$ are terminal for all $t\in S$. Let $W_\infty'\in |2W_\infty|_{\bQ}$ be a sufficiently general $\bQ$-divisor. Then
    \begin{equation}
        \begin{split}
            K_X+\Lambda_X+W_0+W_\infty'+\Pi_{2,X}\sim_{\bQ}H_{2,X}(W_\infty).
        \end{split}
    \end{equation}
   To get rid of the plt center $W_0$, observe that $W_0\sim_{\bQ}W_\infty-A_X$, thus
   \begin{equation}
        \begin{split}
            K_X+\Lambda_X+W_0+W_\infty'+\Pi_{2,X}&\sim_{\bQ}K_X+\tau^*(-K_{W/S})+W_\infty+\tau^*K_{W/S}+W_\infty'+\Pi_{2,X}\\
                                              &\sim_{\bQ}K_X+W_\infty''+\Pi_{2,X},
        \end{split}
    \end{equation}
    where $W_{\infty}''\in |3W_\infty|_{\bQ}$ can be taken general enough so that $(X,D\coloneqq W_\infty''+\Pi_{2,X})$ is still terminal, and the same holds for $(X_t,D_t= W_{\infty,t}''+\Pi_{2,X_t})$ for all $t\in S$.
\end{proof}

\begin{proof}[Proof of \autoref{t-ne_dc}]
     By \autoref{t-ne_cm} and its proof, the contraction of $S$-schemes $\pi_2\colon X\to Y_2$ is the canonical model of $K_X+D$. In particular, it is a $(K_X+D)$-trivial contraction. We now want to replace $D$ with some $G$ so that $(X,G)$ and $(X_t,G_t)$ are still terminal and of general type for all $t\in S$, but $\pi_2$ is now a $(K_X+G)$-negative divisorial contraction. Let $\Pi_1\in |H_1|_{\bQ}$ be a general element, fix $0<\delta<1$ and let $0<\epsilon\ll 1$ such that $3\delta A-\epsilon \Pi_{1,W}$ is ample, and $(X,D+3\delta W_0)$ and $(X_t,D_t+3\delta W_{0,t})$ is terminal for all $t\in S$. Let $E\in |3\delta A-\epsilon \Pi_{1,W}|_{\bQ}$ be a sufficiently general element. Then
     \begin{equation}
         \begin{split}
             D-\epsilon\Pi_{1,X}&\sim_{\bQ} W_\infty''+\Pi_{2,X}-\epsilon\Pi_{1,X}\\
                                     &\sim_{\bQ}(1-\delta)W_\infty''+\Pi_{2,X}+\delta W_\infty''-\epsilon\Pi_{1,X}\\
                                     &\sim_{\bQ}(1-\delta)W_\infty''+\Pi_{2,X}+3\delta W_0+3\delta A_X-\epsilon \Pi_{1,X}\\
                                     &\sim_{\bQ}(1-\delta)W_\infty''+\Pi_{2,X}+3\delta W_0+E_X.
         \end{split}
     \end{equation}
     Let $G\coloneqq (1-\delta)W_\infty''+\Pi_{2,X}+3\delta W_0+E_X$: by construction $(X,G)$ and $(X_t,G_t)$ are terminal and of general type for all $t\in S$, the morphism $\pi_2$ is a $(K_X+G)$-negative divisorial contraction, and $K_{Y_2}+G_2$ is nef. Note also that condition HNCC is satisfied: indeed, the only non-canonical center of $(X,G+X_s)$ is $X_s$ itself, and it is clearly not contained in $\bB_-(K_X+G)=W_0$.
\end{proof}

\begin{proof}[Proof of \autoref{t-ne_flip}]
    By the proof of \autoref{t-ne_cm} we have
    \[
    K_{X}+W_\infty''\sim_{\bQ}W_\infty,
    \]
    where $W_\infty''\in |3W_\infty|_{\bQ}$ is a general element. We can find $0<\varepsilon_0,\varepsilon_2\ll 1$ such that, letting $\Theta_2\in |\varepsilon_2H_{2,X}|_{\bQ}$ be a general element, the $\bQ$-divisor
    \[
    B\coloneqq W_\infty''+\varepsilon_0 W_0+\Theta_2
    \]
    makes $(X,B)$ and $(X_t,B_t)$ terminal for all $t\in S$. Let $B_i\coloneqq \pi_{i,*}B$, and note that $(Y_1,B_1)$ and $(Y_{1,t},B_{1,t})$ are klt for all $t\in S$. Furthermore, by construction, we have that $\psi_1$ and $\psi_2$ are small birational contractions over $S$ of relative Picard rank one, which are $(K_{Y_1}+B_1)$-negative and $(K_{Y_2}+B_2)$-positive, respectively. In particular, the induced map
    \begin{equation}
        \begin{tikzcd}
            (Y_1,B_1)\arrow[dr,"\psi_1",swap]\arrow[rr,"\phi",dashed] &  & (Y_2,B_2)\arrow[dl,"\psi_2"]\\
                           & C & 
        \end{tikzcd}    
    \end{equation}
    is a $(K_{Y_1}+B_1)$-flip over $S$. By \autoref{p-AES_nAES_mckernantrick} we have the following factorization over the closed point of $S$
    \begin{equation}
        \begin{tikzcd}
                                                                                         & ((Y_{2,s})^\nu,(B_{2,s})^\nu)\arrow[dr,"\nu"] &\\ 
            (Y_{1,s},B_{1,s})\arrow[dr,"\psi_{1,s}",swap]\arrow[rr,"\phi_s",dashed]\arrow[ur,"\phi_s^\nu",dashed] &  & (Y_{2,s},B_{2,s})\arrow[dl,"\psi_{2,s}"]\\
                           & C_s, & 
        \end{tikzcd}    
    \end{equation}
    where $\phi_s^\nu$ is a $(K_{Y_{1,s}}+B_{1,s})$-flip, and the demi-normalization $\nu$ is a non-trivial small universal homeomorphism. Note that the only non-canonical centers of $(Y_1,B_1+Y_{1,s})$ are $Y_{1,s}$ and $\pi_1(W_0)$, and only the latter is contained in $\bB_-(K_{Y_1}+B_1)$. As $W_0$ is $\varphi_X$-horizontal, condition HNCC is satisfied. Lastly, note that $K_{Y_2}+B_2$ is nef by construction.
\end{proof}

\begin{proof}[Proof of \autoref{t-failure_AIP_gt}]
    We use the same notation as \autoref{e-MFS} and \autoref{e-birl}. Let $W^0\coloneqq p_2^{-1}(\bA^1_S)\subset W$, let $X^0\coloneqq \tau^{-1}(W^0)\subset X$, and let $Y\coloneqq \pi_2(X^0)\subset Y_2$. This way we have a birational contraction of quasi-projective $S$-schemes
    \begin{center}
        \begin{tikzcd}
            X^0\arrow[r,"\pi"] & Y
        \end{tikzcd}
    \end{center}
    such that $K_{X^0}+W''_\infty|_{X^0}\sim_{Y,\bQ}0$ and $Y_s$ is not $S2$. Note that we also have $K_{X^0}+(l-1)W_\infty|_{X^0}\sim_{Y,\bQ}0$ for all $l\geq 4$. In particular, the restriction map
    \begin{center}
        \begin{tikzcd}
            H^0\left(X^0,\omega_{X^0/S}^m(m(l-1)W_\infty|_{X^0})\right)\arrow[r] & H^0\left(X^0,\omega_{X^0_s/s}^m(m(l-1)W_{\infty}|_{X^0_s})\right)
        \end{tikzcd}
    \end{center}
    fails to be surjective for all $m\geq 1$ divisible enough, by \autoref{l-eq_cond_AES_bigsa}.
    Letting $l$ be odd, a standard $l$-cyclic cover branched over $\Sigma|_{X^0}$, where $\Sigma\in |lW_\infty|$ is a regular divisor, yields a cover $\mu\colon V\to X^0$ such that $V$ is also regular and $\omega_{V/S}\cong\mu^*\omega_{X^0/S}((l-1)W_\infty|_{X^0})$. As $l$ is odd, the line bundle $\omega^m_{X^0/S}(m(l-1)W_\infty|_{X^0})$ is a direct summand of $\mu_*\omega^m_{V/S}$ for all $m$. As a consequence the restriction map
    \begin{center}
        \begin{tikzcd}
            H^0\left(V,\omega_{V/S}^m\right)\arrow[r] & H^0\left(V_s,\omega_{V_s/s}^m\right)
        \end{tikzcd}
    \end{center}
    also fails to be surjective for all $m\geq 1$ divisible enough. Equivalently, if we consider the Stein factorization of $\zeta\coloneqq\pi\circ \mu$
    \begin{center}
        \begin{tikzcd}
            V\arrow[r,"\bar{\zeta}"] & \bar{Y}\arrow[r] & Y,
        \end{tikzcd}
    \end{center}
    we have that $\bar{Y}_s$ is not $S2$, again by \autoref{l-eq_cond_AES_bigsa}. Note that $\bar{\zeta}$ is still a birational contraction of $S$-schemes, and that we have $K_V\sim_{\bar{Y},\bQ}0$. By ``spreading out'' the above morphism over $U\otimes\bF_2\llbracket s \rrbracket$ we then obtain a birational contraction of flat $\bF_2\llbracket s\rrbracket$-schemes
    \begin{center}
        \begin{tikzcd}
            \cV\arrow[r,"\widetilde{\bar{\zeta}}"] & \bar{\cY}
        \end{tikzcd}
    \end{center}
    satisfying the requirements of the theorem. 
\end{proof}

\bibliographystyle{alpha}
\bibliography{bib.bib}

\end{document}